\documentclass{amsart}
\usepackage{amssymb}
\usepackage{mathrsfs}
\usepackage{cases}
\usepackage{amsmath}
\usepackage{amsfonts}
\usepackage{pifont}
\usepackage{graphicx,amssymb,mathrsfs,amsmath}
\usepackage{tocvsec2}
\newtheorem{thm}{Theorem}[section]
\newtheorem{cor}[thm]{Corollary}
\newtheorem{lem}[thm]{Lemma}

\theoremstyle{definition}
\newtheorem{defn}[thm]{Definition}
\newtheorem{example}[thm]{Example}
\theoremstyle{remark}
\newtheorem{rem}[thm]{Remark}
\numberwithin{equation}{section}

\begin{document}

\title{ASYMPTOTICALLY ALMOST PERIODIC 
SOLUTIONS OF FRACTIONAL RELAXATION INCLUSIONS WITH CAPUTO DERIVATIVES}

\author{Marko Kosti\' c}
\address{Faculty of Technical Sciences,
University of Novi Sad,
Trg D. Obradovi\' ca 6, 21125 Novi Sad, Serbia}
\email{marco.s@verat.net}

\begin{abstract}
In the paper under review, we analyze asymptotically almost periodic solutions
for a class of (semilinear) fractional relaxation inclusions with Stepanov almost periodic coefficients. As auxiliary tools, we use subordination principles, fixed point theorems and the well known results on the generation of infinitely differentiable degenerate semigroups with removable singularites at zero. Our results are well illustrated and seem to be not considered elsewhere even for fractional relaxation equations with almost sectorial operators.
\end{abstract}

{\renewcommand{\thefootnote}{} \footnote{2010 {\it Mathematics
Subject Classification. 34G25, 47D03, 47D06, 47D99.
\\ \text{  }  \ \    {\it Key words and phrases.} Abstract semilinear Cauchy inclusions,  Asymptotic almost periodicity, Stepanov asymptotic almost periodicity.
\\  \text{  } The author is partially supported by grant 174024 of Ministry
of Science and Technological Development, Republic of Serbia.}}

\maketitle

\section{Introduction and preliminaries}

The notion of an almost periodic function was introduced by Bohr in 1925 and later generalized by many other mathematicians (see e.g. \cite{diagana}, \cite{gaston}, \cite{hino-bor} and \cite{18}).
Let $I={\mathbb R}$ or $I=[0,\infty),$ and let $f : I \rightarrow X$ be continuous. Given $\epsilon>0,$ we call $\tau>0$ an $\epsilon$-period for $f(\cdot)$ iff
$
\| f(t+\tau)-f(t) \| \leq \epsilon, $ $ t\in I.
$
The set consisted of all $\epsilon$-periods for $f(\cdot)$ is denoted by $\vartheta(f,\epsilon).$ It is said that $f(\cdot)$ is almost periodic, a.p. for short, iff for each $\epsilon>0$ the set $\vartheta(f,\epsilon)$ is relatively dense in $I,$ which means that
there exists $l>0$ such that any subinterval of $I$ of length $l$ meets $\vartheta(f,\epsilon)$. The space consisted of all almost periodic functions from the interval $I$ into $X$ will be denoted by $AP(I:X).$

The notion of an asymptotically almost periodic function was introduced by Fr\' echet in 1941 (for further information concerning the vector-valued asymptotically almost periodic functions, see \cite{cheban}-\cite{diagana}, \cite{gaston} and references cited therein). A function $f \in C_{b}([0,\infty) : X)$ is called asymptotically almost periodic iff
for every $\epsilon >0$ we can find numbers $ l > 0$ and $M >0$ such that every subinterval of $[0,\infty)$ of
length $l$ contains, at least, one number $\tau$ such that $\|f(t+\tau)-f(t)\| \leq \epsilon$ for all $t \geq M.$
The space consisting of all asymptotically almost periodic functions from $[0,\infty)$ into $X$ is denoted by
$AAP([0,\infty) : X).$ For a function $f \in C([0,\infty):X),$ the following
statements are equivalent (\cite{RUESS}):
\begin{itemize}
\item[(i)] $f\in AAP([0,\infty) :X).$
\item[(ii)] There exist uniquely determined functions $g \in AP([0,\infty) :X)$ and $\phi \in  C_{0}([0,\infty): X)$
such that $f = g+\phi.$
\item[(iii)] The set $H(f):=\{f(\cdot +s) : s\geq 0\}$ is relatively compact in $C_{b}([0,\infty):X).$
\end{itemize}

Assume that $1\leq p <\infty.$ Then it is said that a function $f\in L^{p}_{loc}(I :X)$ is Stepanov $p$-bounded, $S^{p}$-bounded shortly, iff
$$
\|f\|_{S^{p}}:=\sup_{t\in I}\Biggl( \int^{t+1}_{t}\|f(s)\|^{p}\, ds\Biggr)^{1/p}<\infty.
$$
The space $L_{S}^{p}(I:X)$ consisting of all $S^{p}$-bounded functions becomes a Banach space when equipped with the above norm. 
A function $f\in L_{S}^{p}(I:X)$ is said to be Stepanov $p$-almost periodic, $S^{p}$-almost periodic shortly, iff the function
$
\hat{f} : I \rightarrow L^{p}([0,1] :X),
$ defined by 
$$
\hat{f}(t)(s):=f(t+s),\quad t\in I,\ s\in [0,1]
$$
is almost periodic (cf. Amerio, Prouse \cite{amerio} for more details). 
We say that $f\in  L_{S}^{p}([0,\infty): X)$ is asymptotically Stepanov $p$-almost periodic, asymptotically $S^{p}$-almost periodic shortly, iff $\hat{f} : [0,\infty) \rightarrow L^{p}([0,1]:X)$ is asymptotically almost periodic. We use the shorthands $APS^{p}([0,\infty) : X)$ and $AAPS^{p}([0,\infty) : X)$ to denote the vector spaces consisting of all Stepanov $p$-almost periodic functions and asymptotically Stepanov $p$-almost periodic functions, respectively.
It is well-known that if $f(\cdot)$ is an almost periodic (respectively, a.a.p.) function
then $f(\cdot)$ is also $S^p$-almost periodic (resp., $S^p$-a.a.p.) for $1\leq p <\infty.$ The converse statement is not true, in general. 

Concerning almost periodic and asymptotically almost periodic solutions of various classes of abstract Volterra integro-differential equations in Banach spaces, the reader may consult \cite{amerio}, \cite{cheban}-\cite{diagana}, \cite{gaston}-\cite{hino-bor}, \cite{AOT}-\cite{relaxation-peng} and \cite{zhang-c-prim}.

Let $\gamma \in (0,1),$ and let ${\mathcal A}$ be a multivalued linear operator on a Banach space $X.$
Of importance is the following fractional relaxation inclusion
\[
\hbox{(DFP)}_{f,\gamma} : \left\{
\begin{array}{l}
{\mathbf D}_{t}^{\gamma}u(t)\in {\mathcal A}u(t)+f(t),\ t> 0,\\
\quad u(0)=x_{0},
\end{array}
\right.
\]
and its semilinear analogue
\[
\hbox{(DFP)}_{f,\gamma,s} : \left\{
\begin{array}{l}
{\mathbf D}_{t}^{\gamma}u(t)\in {\mathcal A}u(t)+f(t,u(t)),\ t> 0,\\
\quad u(0)=x_{0},
\end{array}
\right.
\]
where ${\mathbf D}_{t}^{\gamma}$ denotes the Caputo fractional derivative of order $\gamma,$ $x_{0}\in X$ and 
$f : [0,\infty) \rightarrow X,$ resp.
$f : [0,\infty) \times X \rightarrow X,$ is Stepanov almost periodic. 
The main aim of this paper is to continue our recent research studies \cite{AOT}-\cite{klavir} by investigating asymptotically almost periodic solutions
of fractional Cauchy inclusions (DFP)$_{f,\gamma}$ and (DFP)$_{f,\gamma,s}.$ We would like to note that the existence and uniqueness of (asymptotically) quasi-periodic solutions of fractional relaxation equations with Caputo derivatives have not received much attention so far: in the existing literature,
we have been able to find only one research paper (Li, Liang, Wang \cite{relaxation-liang}) concerning similar problematic; that paper is devoted to the study of $S$-asymptotically
$\omega$-periodic solutions to fractional relaxation equations with finite
delay. The established results of ours are completely new in degenerate case and, as already mentioned in the abstract, they seem to be new even for fractional relaxation equations with almost sectorial operators (\cite{pb1}). 

The paper is very simply organized and we deeply believe that
our readers will fairly quickly move throughout it.
We use the standard notation henceforth.
By $X$ and $Y$ we denote two Banach spaces over the field of complex numbers. 
The symbol
$L(X,Y)$ stands for the space consisting of all continuous linear mappings from $X$ into
$Y;$ $L(X)\equiv L(X,X).$ By $I$ we denote the identity operator on $X.$
Denote by $C_{b}([0,\infty):X),$ $C_{0}([0,\infty):X)$ and $BUC([0,\infty):X)$ the vector spaces consisted of all bounded continuous functions from $[0,\infty)$ into $X,$ all bounded continuous functions from $[0,\infty)$ into $X$ vanishing at infinity, and  all bounded uniformly continuous functions from $[0,\infty)$
into $X,$ respectively. Equipped with the usual sup-norm, any of these spaces becomes one of Banach's. 

Denote by $C_{0}([0,\infty) \times Y : X)$ the space of all
continuous functions $h : [0,\infty)  \times Y \rightarrow X$ satisfying that $\lim_{t\rightarrow \infty}h(t, y) = 0$ uniformly for $y$ in any compact subset of $Y .$  If $f : I  \times Y  \rightarrow X,$ then we define $\hat{f} : I  \times Y  \rightarrow L^{p}([0,1]:X)$ by $\hat{f}(t , y):=f(t +\cdot, y),$ $t\geq 0,$ $y\in Y.$

For the purpose of research of asymptotically almost periodic solutions of semilinear fractional Cauchy inclusions, we need to recall the following well-known definitions (see e.g. Zhang \cite{zhang-c-prim}, Long, Ding \cite{comp-adv}, and Kosti\' c \cite{hjms}):

\begin{defn}\label{definicija}
Let $1\leq p <\infty.$ 
\begin{itemize}
\item[(i)]
A function $f : I \times Y \rightarrow X$ is said to be almost periodic iff $f (\cdot, \cdot)$ is bounded, continuous as well as for every $\epsilon>0$ and every compact
$K\subseteq Y$ there exists $l(\epsilon,K) > 0$ such that every subinterval $J\subseteq I$ of length $l(\epsilon,K)$ contains a number $\tau$ with the property that $\|f (t +\tau , y)- f (t, y)\| \leq \epsilon$ for all $t \in  I,$ $ y \in K.$ The collection of such functions will be denoted by $AP(I \times Y : X).$
\item[(ii)] A function $f : [0,\infty)  \times Y \rightarrow X$ is said to be asymptotically almost periodic iff it is bounded continuous and admits a
decomposition $f = g + q,$ where $g \in AP({\mathbb R}  \times Y : X)$ and $q\in C_{0}([0,\infty)  \times Y : X).$ Denote by
 $AAP([0,\infty)  \times Y : X) $ the vector space consisting of all such functions.
\item[(iii)] A function $f : I \times Y \rightarrow X$ is called Stepanov $p$-almost periodic, $S^{p}$-almost periodic shortly, iff $\hat{f} : I  \times Y  \rightarrow L^{p}([0,1]:X)$ is almost periodic.
\item[(iv)] A function $f : [0,\infty)  \times Y \rightarrow X$
is said to be asymptotically $S^p$-almost periodic
iff $\hat{f}: [0,\infty)  \times Y \rightarrow  L^{p}([0,1]:X)$ is asymptotically almost periodic. The collection of such functions will be denoted by $AAPS^{p}([0,\infty) \times Y : X).$
\end{itemize}
\end{defn}

It could be of importance to remind ourselves of the following known facts (\cite{zhang-c-prim}): 

\begin{itemize}
\item[(i)]
Let $f \in AP(I \times Y : X)$ and $h \in AP(I : Y ).$ Then the mapping $t\mapsto f (t,h(t)),$ $t\in I$  belongs to the space $ AP(I:X).$ 
\item[(ii)] Let $f \in AAP([0,\infty) \times Y : X)$ and $h \in AAP([0,\infty) : Y ).$ Then the mapping $t\mapsto f (t,h(t)),$ $t\geq 0$ belongs to the space $AAP([0,\infty) : X).$
\end{itemize}

It can be easily proved that any asymptotically almost periodic two-parameter function is also asymptotically Stepanov $p$-almost periodic
($1\leq p <\infty$). More details about Stepanov $p$-almost periodic functions depending on two parameters can be found in \cite{hjms}.

Given $s\in {\mathbb R}$ in advance, set $\lfloor s \rfloor :=\sup \{
l\in {\mathbb Z} : s\geq l \}$ and $\lceil s \rceil:=\inf \{
l\in {\mathbb Z} : s\leq l \}.$ The Gamma function is denoted by
$\Gamma(\cdot)$ and the principal branch is always used to take the powers. Define $g_{\alpha}(t):=t^{\alpha -1}/\Gamma(\alpha),$ $t>0$ ($\alpha>0$).

Fractional calculus and fractional differential equations are rapidly growing fields of research (see e.g. \cite{bajlekova}, \cite{Diet}, \cite{knjigah}-\cite{knjigaho}, \cite{prus}, \cite{samko} and \cite{fractionalsectorial}).
Let $\gamma \in (0,1).$ Then the Wright function $\Phi_{\gamma}(\cdot)$ is 
defined by the formula
$$
\Phi_{\gamma}(z):=\sum \limits_{n=0}^{\infty}
\frac{(-z)^{n}}{n! \Gamma (1-\gamma -\gamma n)},\quad z\in {\mathbb C}.
$$
It is well known that $\Phi_{\gamma}(\cdot)$ is an entire function, as well as that $\Phi_{\gamma}(t)\geq 0,$ $t\geq 0,$
$\int^{\infty}_{0}t^{r}
\Phi_{\gamma}(t)\, dt=\frac{\Gamma(1+r)}{\Gamma(1+\gamma r)},$ $r>-1$ and $\int^{\infty}_{0}e^{-z t}\Phi_{\gamma}(t)\, dt=E_{\gamma}(-z),$ $z\in {\mathbb C},$
where $E_{\gamma}(\cdot)$ denotes the Mittag-Leffler function.
For more details about the Mittag-Leffler and Wright functions, we refer the reader to the doctoral dissertation of Bazhlekova \cite{bajlekova} and references cited therein.

Let $0<\tau \leq \infty,$ let $m\in {\mathbb N},$ and let
$I=(0,\tau).$ Then we can introduce the Sobolev space $W^{m,1}(I : X)$ in the following way (see e.g. \cite[p. 7]{bajlekova}):
\begin{align*}
W^{m,1}(I :X):=\Biggl \{ f \ | & \ \exists \varphi \in L^{1}(I : X)
\ \exists c_{k}\in {\mathbb C} \ (0\leq k\leq m-1) \\ &
f(t)=\sum_{k=0}^{m-1}c_{k}g_{k+1}(t)+\bigl(g_{m}\ast \varphi
\bigr)(t) \mbox{ for a.e. }t\in (0,\tau)\Biggr\}.
\end{align*}
If this is the case, we have $\varphi(t)=f^{(m)}(t)$ in distributional sense, and
$c_{k}=f^{(k)}(0)$ ($0\leq k\leq m-1$). 

\subsection{Multivalued linear operators in Banach spaces.}
The multivalued linear operators approach to abstract degenerate differential equations with integer order derivatives has been obeyed in the 
fundamental monograph \cite{faviniyagi} by Favini and Yagi (cf. also Kamenskii, Obukhovskii and Zecca \cite{kamenski} for a slightly different approach to
condensing multivalued mappings and semilinear
differential inclusions in Banach spaces).
In what follows, we will present a brief overview of definitions from the theory of multivalued linear operators in Banach spaces. 

A multivalued map (multimap) ${\mathcal A} : X \rightarrow P(Y)$ is said to be a multivalued
linear operator (MLO) iff the following holds:
\begin{itemize}
\item[(i)] $D({\mathcal A}) := \{x \in X : {\mathcal A}x \neq \emptyset\}$ is a linear subspace of $X$;
\item[(ii)] ${\mathcal A}x +{\mathcal A}y \subseteq {\mathcal A}(x + y),$ $x,\ y \in D({\mathcal A})$
and $\lambda {\mathcal A}x \subseteq {\mathcal A}(\lambda x),$ $\lambda \in {\mathbb C},$ $x \in D({\mathcal A}).$
\end{itemize}
If $X=Y,$ then we say that ${\mathcal A}$ is an MLO in $X.$

It is well known that, if $x,\ y\in D({\mathcal A})$ and $\lambda,\ \eta \in {\mathbb C}$ with $|\lambda| + |\eta| \neq 0,$ then 
$\lambda {\mathcal A}x + \eta {\mathcal A}y = {\mathcal A}(\lambda x + \eta y).$ Assuming ${\mathcal A}$ is an MLO, we have that ${\mathcal A}0$ is a linear subspace of $Y$
and ${\mathcal A}x = f + {\mathcal A}0$ for any $x \in D({\mathcal A})$ and $f \in {\mathcal A}x.$ Define $R({\mathcal A}):=\{{\mathcal A}x :  x\in D({\mathcal A})\}.$
Then the set ${\mathcal A}^{-1}0 = \{x \in D({\mathcal A}) : 0 \in {\mathcal A}x\}$ is called the kernel
of ${\mathcal A}$ and it is denoted by $N({\mathcal A}).$ The inverse ${\mathcal A}^{-1}$ of an MLO is defined by
$D({\mathcal A}^{-1}) := R({\mathcal A})$ and ${\mathcal A}^{-1} y := \{x \in D({\mathcal A}) : y \in {\mathcal A}x\}$.
It can be simply shown that ${\mathcal A}^{-1}$ is an MLO in $X,$ as well as that $N({\mathcal A}^{-1}) = {\mathcal A}0$
and $({\mathcal A}^{-1})^{-1}={\mathcal A};$ ${\mathcal A}$ is said to be injective iff ${\mathcal A}^{-1}$ is
single-valued. 

Let ${\mathcal A},\ {\mathcal B} : X \rightarrow P(Y)$ be two MLOs. Then we define its sum ${\mathcal A}+{\mathcal B}$ by $D({\mathcal A}+{\mathcal B}) := D({\mathcal A})\cap D({\mathcal B})$ and $({\mathcal A}+{\mathcal B})x := {\mathcal A}x +{\mathcal B}x,$ $x\in D({\mathcal A}+{\mathcal B}).$ Clearly, ${\mathcal A}+{\mathcal B}$ is likewise an MLO.

Assume that ${\mathcal A} : X \rightarrow P(Y)$ and ${\mathcal B} : Y\rightarrow P(Z)$ are two MLOs, where $Z$ is likewise a complex Banach space. The product of operators ${\mathcal A}$
and ${\mathcal B}$ is defined by $D({\mathcal B}{\mathcal A}) :=\{x \in D({\mathcal A}) : D({\mathcal B})\cap {\mathcal A}x \neq \emptyset\}$ and
${\mathcal B}{\mathcal A}x:=
{\mathcal B}(D({\mathcal B})\cap {\mathcal A}x).$ It is well known that ${\mathcal B}{\mathcal A} : X\rightarrow P(Z)$ is an MLO and
$({\mathcal B}{\mathcal A})^{-1} = {\mathcal A}^{-1}{\mathcal B}^{-1}.$ 

It is said that an MLO operator  ${\mathcal A} : X\rightarrow P(Y)$ is closed if for any
sequences $(x_{n})$ in $D({\mathcal A})$ and $(y_{n})$ in $Y$ such that $y_{n}\in {\mathcal A}x_{n}$ for all $n\in {\mathbb N}$ we have that the suppositions $\lim_{n \rightarrow \infty}x_{n}=x$ and
$\lim_{n \rightarrow \infty}y_{n}=y$ imply
$x\in D({\mathcal A})$ and $y\in {\mathcal A}x.$

We need the following auxiliary lemma from \cite{FKP}.

\begin{lem}\label{integracija-tricky}
Let $\Omega$ be a locally compact, separable metric space,
and let $\mu$ be a locally finite
Borel measure defined on $\Omega.$
Suppose that ${\mathcal A} : X\rightarrow P(Y)$ is a closed \emph{MLO}. Let $f : \Omega \rightarrow X$ and $g : \Omega \rightarrow Y$ be $\mu$-integrable, and let $g(x)\in {\mathcal A}f(x),$ $x\in \Omega.$ Then $\int_{\Omega}f\, d\mu \in D({\mathcal A})$ and $\int_{\Omega}g\, d\mu\in {\mathcal A}\int_{\Omega}f\, d\mu.$
\end{lem}

Let ${\mathcal A}$ be an MLO in $X,$ let $C\in L(X)$ be injective, and let $C{\mathcal A}\subseteq {\mathcal A}C.$ 
Then
the $C$-resolvent set of ${\mathcal A},$ $\rho_{C}({\mathcal A})$ for short, is defined as the union of those complex numbers
$\lambda \in {\mathbb C}$ for which
\begin{itemize}
\item[(i)] $R(C)\subseteq R(\lambda-{\mathcal A})$;
\item[(ii)] $(\lambda - {\mathcal A})^{-1}C$ is a single-valued linear continuous operator on $X.$
\end{itemize}
The operator $\lambda \mapsto (\lambda -{\mathcal A})^{-1}C$ is called the $C$-resolvent of ${\mathcal A}$ ($\lambda \in \rho_{C}({\mathcal A})$); the resolvent set of ${\mathcal A}$ is defined by $\rho({\mathcal A}):=\rho_{I}({\mathcal A}),$ $R(\lambda : {\mathcal A})\equiv  (\lambda -{\mathcal A})^{-1}$  ($\lambda \in \rho({\mathcal A})$). 
The basic properties of $C$-resolvents of single-valued linear operators continue to hold in the multivalued linear setting  (\cite{faviniyagi}, \cite{FKP}).

Concerning the abstract degenerate Volterra integro-differential equations and abstract degenerate fractional differential equations, the reader may consult author's forthcoming monograph \cite{FKP}.

Now we will recall the basic things about fractional powers and interpolation spaces of multivalued linear operators.
Assume that $X$ is a Banach space, $(-\infty,0]\subseteq \rho({\mathcal A})$ as well as there exist finite numbers $M\geq 1$ and $\beta \in (0,1]$ such that
$
\| R(\lambda : {\mathcal A}) \| \leq M(1+|\lambda|)^{-\beta},$ $ \lambda \leq 0.
$
Then it is not difficult to see that there exist two positive real constants
$c>0$ and $M_{1}>0$
such that
$\rho({\mathcal A})$ contains an open region $\Omega =\{
\lambda \in {\mathbb C} : |\Im \lambda| \leq  (2M_{1})^{-1} (c-\Re \lambda)^{\beta},\ \Re \lambda \leq c\}$ of complex plane around the half-line $(-\infty,0],$ where we have the estimate
$
\| R(\lambda : {\mathcal A}) \| =O((1+|\lambda|)^{-\beta}),$ $ \lambda \in \Omega .$ Designate by $\Gamma '$ the upwards oriented curve
$\{ \xi \pm  i(2M_{1})^{-1}(c-\xi)^{\beta} : -\infty <\xi \leq c \}.$
Following Favini and Yagi \cite{faviniyagi}, we define the fractional power
$$
{\mathcal A}^{-\theta}:=\frac{1}{2\pi i}\int_{\Gamma'}\lambda^{-\theta} \bigl( \lambda-{\mathcal A} \bigr)^{-1}\, d\lambda \in L(X)
$$
for $\theta >1-\beta$. Set ${\mathcal A}^{\theta}:=({\mathcal A}^{-\theta})^{-1}$ ($\theta >1-\beta$). Then the semigroup properties ${\mathcal A}^{-\theta_{1}}{\mathcal A}^{-\theta_{2}}={\mathcal A}^{-(\theta_{1}+\theta_{2})}$ and ${\mathcal A}^{\theta_{1}}{\mathcal A}^{\theta_{2}}={\mathcal A}^{\theta_{1}+\theta_{2}}$
hold for $\theta_{1},\ \theta_{2}>1-\beta$ (recall that the fractional power ${\mathcal A}^{\theta}$ need not be injective and the meaning of ${\mathcal A}^{\theta}$ is understood in the MLO sense for $\theta >1-\beta$).

We topologize the vector space $D({\mathcal A})$ by the norm
$
\|\cdot\|_{[D({\mathcal A})]}:=\inf_{y\in {\mathcal A} \cdot}\|y\|.
$
Then $(D({\mathcal A}),\|\cdot \|_{[D({\mathcal A})]})$ is a Banach space and the norm  $
\|\cdot\|_{[D({\mathcal A})]}$ is equivalent with the following one $\|\cdot\|+
\|\cdot\|_{[D({\mathcal A})]};$ $(D({\mathcal A}^{\theta}),\|\cdot \|_{[D({\mathcal A}^{\theta})]})$ is likewise a Banach space and we have the equivalence of
norms $
\|\cdot\|_{[D({\mathcal A}^{\theta})]}$ and $\|\cdot\|+
\|\cdot\|_{[D({\mathcal A}^{\theta})]}$ for $\theta > 1-\beta$ (cf. the proof of \cite[Proposition 1.1]{faviniyagi}).

Let $\theta \in (0,1),$ let
$$
X_{{\mathcal A}}^{\theta}:=\Biggl\{  x\in X : \sup_{\xi >0}\xi^{\theta}\Bigl\| \xi \bigl(\xi +{\mathcal A}\bigr)^{-1}x- x\Bigr\|<\infty \Biggr\},
$$
and let
$$
\bigl\| \cdot \bigr\|_{X_{{\mathcal A}}^{\theta}}:=\|\cdot\|+\sup_{\xi >0}\xi^{\theta}\Bigl \| \xi \bigl(\xi +{\mathcal A}\Bigr)^{-1}\cdot -\cdot  \Bigr\|.
$$
Then it is well known that $
(X_{{\mathcal A}}^{\theta},\| \cdot \|_{X_{{\mathcal A}}^{\theta}})$ is a Banach space as well as that $X_{{\mathcal A}}^{\theta}$ continuously embedded in $
X.$

We refer the reader to \cite{tsk}-\cite{faviniyagi} and \cite{FKP} for further information concerning interpolation spaces and fractional powers of multivalued linear operators. 

\section{Asymptotically almost periodic solutions of abstract fractional Cauchy inclusions (DFP)$_{f,\gamma}$ and (DFP)$_{f,\gamma,s}$}\label{idiot}

In order to formulate our main results, we need to recall two important lemmae regarding the composition principles for asymptotically Stepanov almost periodic two-parameter functions (\cite{hjms}). 

\begin{lem}\label{bibl}
Let $I =[0,\infty).$
Suppose that the following conditions hold:
\begin{itemize}
\item[(i)] $g \in APS^{p}(I \times X : X)  $ with  $p > 1, $ and there exist a number  $ r\geq \max (p, p/p -1)$ and a function $ L_{g}\in L_{S}^{r}(I:X) $ such that
\begin{align}\label{vbnmp}
\| g(t,x)-g(t,y)\| \leq L_{g}(t)\|x-y\|,\quad t\geq 0,\ x,\ y\in X.
\end{align}
\item[(ii)] $y \in APS^{p}(I:X),$ and there exists a Lebesgue's measurable set $E \subseteq I$ with $m (E)= 0$ such that
$ K =\{y(t) : t \in I \setminus E\}$
is relatively compact in X.
\item[(iii)] $f(t,x)=g(t,x)+q(t,x)$ for all $t\geq 0$ and $x\in X,$ where $\hat{q}\in C_{0}([0,\infty) \times X : L^{q}([0,1]:X))$
and $q:=pr/p+r.$
\item[(iv)] $x(t)=y(t)+z(t) $ for all $t\geq 0,$ where $\hat{z}\in C_{0}([0,\infty) : L^{p}([0,1]:X)).$
\item[(v)]  There exists a Lebesgue's measurable set $E' \subseteq I$ with $m (E')= 0$ such that
$ K' =\{x(t) : t \in I \setminus E'\}$
is relatively compact in $ X.$
\end{itemize}
Then $q\in [1, p)$ and $f(\cdot, x(\cdot)) \in AAPS^{q}(I : X).$
\end{lem}

In the case of consideration of usual Lipschitz type condition 
\begin{align}\label{vbnmp-frim}
\| f(t,x)-f(t,y)\| \leq L\|x-y\|,\quad t\geq 0,\ x,\ y\in X,
\end{align}
the following result holds true:

\begin{lem}\label{bibl-frimaonji}
Let $I =[0,\infty).$
Suppose that the following conditions hold:
\begin{itemize}
\item[(i)] $g \in APS^{p}(I \times X : X)  $ with  $p \geq 1, $ and there exists a constant $L>0$ such that \eqref{vbnmp-frim} holds with the function $f(\cdot, \cdot)  $ replaced by the function $g(\cdot, \cdot)  $ therein.
\item[(ii)] $y \in APS^{p}(I:X),$ and there exists a Lebesgue's measurable set $E \subseteq I$ with $m (E)= 0$ such that
$ K =\{y(t) : t \in I \setminus E\}$
is relatively compact in X.
\item[(iii)] $f(t,x)=g(t,x)+q(t,x)$ for all $t\geq 0$ and $x\in X,$ where $\hat{q}\in C_{0}([0,\infty) \times X : L^{p}([0,1]:X)).$
\item[(iv)] $x(t)=y(t)+z(t) $ for all $t\geq 0,$ where $\hat{z}\in C_{0}([0,\infty) : L^{p}([0,1]:X)).$
\item[(v)] There exists a Lebesgue's measurable set $E' \subseteq I$ with $m (E')= 0$ such that
$ K' =\{x(t) : t \in I \setminus E'\}$
is relatively compact in $ X.$
\end{itemize}
Then $f(\cdot, x(\cdot)) \in AAPS^{p}(I : X).$
\end{lem}

The following lemma can be proved in exactly the same way as \cite[Proposition 2.11]{EJDE}.

\begin{lem}\label{ravi-and}
Suppose that $1\leq p <\infty,$ $1/p +1/q=1$
and $(R(t))_{t> 0}\subseteq L(X,Y)$ is a strongly continuous operator family satisfying that $M:=\sum_{k=0}^{\infty}\|R(\cdot)\|_{L^{q}[k,k+1]}<\infty .$ If $f : {\mathbb R} \rightarrow X$ is $S^{p}$-almost periodic, then the function $G: {\mathbb R} \rightarrow Y,$ given by
\begin{align}\label{wer}
G(t):=\int^{t}_{-\infty}R(t-s)f(s)\, ds,\quad t\in {\mathbb R},
\end{align}
is well-defined and almost periodic.
\end{lem}

\begin{rem}\label{conclusion}
Let $p>1,$ and let $t\mapsto \|R(t)\|,$ $t\in (0,1]$ be an element of the space $L^{q}[0,1].$
Then
the inequality $\sum_{k=0}^{\infty}\|R(\cdot)\|_{L^{q}[k,k+1]}<\infty $ holds provided that there exists a finite number $\zeta <0$ such that $\|R(t)\| =O(t^{\zeta}),$ $t\rightarrow +\infty$ and
$\zeta <(1/p)-1.$
\end{rem}

\subsection{Subordinated fractional resolvent families with removable singularities at zero}\label{tarantino}

In this subsection, we analyze the class of multivalued linear operators ${\mathcal A}$ satisfying the condition \cite[(P), p. 47]{faviniyagi}:
\begin{itemize} \index{removable singularity at zero}
\item[(P)]
There exist finite constants $c,\ M>0$ and $\beta \in (0,1]$ such that\index{condition!(PW)}
$$
\Psi:=\Psi_{c}:=\Bigl\{ \lambda \in {\mathbb C} : \Re \lambda \geq -c\bigl( |\Im \lambda| +1 \bigr) \Bigr\} \subseteq \rho({\mathcal A})
$$
and
$$
\| R(\lambda : {\mathcal A})\| \leq M\bigl( 1+|\lambda|\bigr)^{-\beta},\quad \lambda \in \Psi.
$$
\end{itemize}

Suppose that the condition (P) holds.
Then degenerate strongly continuous semigroup $(T(t))_{t> 0}\subseteq L(X)$ generated by ${\mathcal A}$ satisfies estimate
$\|T(t) \| \leq M_{0} e^{-ct}t^{\beta -1},$ $t> 0$ for some finite constant $M_{0}>0$ (\cite{EJDE}). Furthermore, $(T(t))_{t> 0}$ is given by the formula
$$
T(t)x=\frac{1}{2\pi i}\int_{\Gamma}e^{\lambda t}\bigl(  \lambda -{\mathcal A} \bigr)^{-1}x\, d\lambda,\quad t>0,\ x\in X,
$$
where $\Gamma$ is the upwards oriented curve $\lambda=-c(|\eta|+1)+i\eta$ ($\eta \in {\mathbb R}$).
Let $0<\gamma <1,$ and let $\nu >-\beta .$ Set
\begin{align*}
T_{\gamma,\nu}(t)x:=t^{\gamma \nu}\int^{\infty}_{0}s^{\nu}\Phi_{\gamma}( s)T\bigl( st^{\gamma}\bigr)x\, ds,\quad t>0,\ x\in X.
\end{align*}
Following Bazhlekova \cite{bajlekova} and Wang, Chen, Xiao \cite{fractionalsectorial}, we set
$$
S_{\gamma}(t):=T_{\gamma,0}(t)\mbox{ and }P_{\gamma}(t):=\gamma T_{\gamma,1}(t)/t^{\gamma},\quad t>0.
$$
Recall that $(S_{\gamma}(t))_{t>0}$ is a subordinated $(g_{\gamma},I)$-regularized resolvent family generated by ${\mathcal A},$ which is not necessarily strongly continuous at zero. 
In \cite{FKP}, we have proved that there exists a finite constant $M_{1}>0$ such that
\begin{align}\label{debil}
\bigl\| S_{\gamma}(t) \bigr\|+\bigl\| P_{\gamma}(t) \bigr\|\leq M_{1}t^{\gamma (\beta-1)},\quad t>0.
\end{align}

For our later purposes, we need to improve the growth order of these operator families at infinity:

\begin{lem}\label{nmb}
There exists a finite constant $M_{2}>0$ such that
\begin{align}\label{debil-prim}
\bigl\| S_{\gamma}(t) \bigr\|\leq M_{2}t^{-\gamma},\ t\geq 1 \ \ \mbox{  and  }\ \ \bigl\| P_{\gamma}(t) \bigr\|\leq M_{2}t^{-2\gamma },\ t\geq 1.
\end{align}
\end{lem}

\begin{proof}
By definition of $(S_{\gamma}(t))_{t>0}$ and growth order of $(T(t))_{t> 0},$ we have:
\begin{align*}
\bigl \| S_{\gamma}(t) \bigr\| \leq M_{0}\int^{\infty}_{0}\Phi_{\gamma}( s)e^{cst^{\gamma}}\bigl( st^{\gamma}\bigr)^{\beta-1}\, ds
=M_{0}t^{\gamma (\beta -1)}\int^{\infty}_{0}e^{cst^{\gamma}}\Phi_{\gamma}( s)s^{\beta-1}\, ds,
\end{align*} 
for any $t>0.$ Since $\Phi_{\gamma}( s) \sim (\Gamma (1-\gamma))^{-1} ,$ $s\rightarrow 0+,$ a Tauberian type theorem \cite[Proposition 4.1.4; b)]{a43} immediately implies the first estimate in \eqref{debil-prim}. We can prove the second estimate in \eqref{debil-prim}
by using the same result, since
$$
\bigl \| P_{\gamma}(t) \bigr\| \leq M_{0}t^{\gamma (\beta -1)}\int^{\infty}_{0}e^{cst^{\gamma}}\Phi_{\gamma}(s) s^{\beta}\, ds,\quad t>0.
$$
\end{proof}

Suppose that $x_{0}\in X$ belongs to the domain of continuity of $(S_{\gamma}(t))_{t>0},$ that is $\lim_{t\rightarrow 0+}S_{\gamma}(t)x_{0}=x_{0}$ (this holds provided that $x_{0}\in X$ belongs to the space $D((-{\mathcal A})^{\theta})$ with $1\geq \theta >1-\beta$ or that $x \in X_{{\mathcal A}}^{\theta}$ with $1>\theta >1-\beta$).

In this subsection, we will employ the following definition of Caputo fractional derivatives of order $\gamma \in (0,1)$. 
The Caputo fractional derivative\index{fractional derivatives!Caputo}
${\mathbf D}_{t}^{\gamma}u(t)$ is defined for those functions
$u :[0,T] \rightarrow X$ for which $u _{| (0,T]}(\cdot) \in C((0,T]: E),$ $u(\cdot)-u(0) \in L^{1}((0,T) : X)$
and $g_{1-\gamma}\ast (u(\cdot)-u(0)) \in W^{1,1}((0,T) : X),$
by
$$
{\mathbf
D}_{t}^{\gamma}u(t)=\frac{d}{dt}\Biggl[g_{1-\gamma}
\ast \Bigl(u(\cdot)-u(0)\Bigr)\Biggr](t),\quad t\in (0,T].
$$

Morover, we will use the following definition (cf. \cite[Section 3.5]{FKP} for more details on the subject):

\begin{defn}\label{defin-ska-MLOs}
By a classical solution of (DFP)$_{f,\gamma},$ we mean any function $u\in C([0,\infty) : X)$ satisfying that the function
${\mathbf D}_{t}^{\gamma}u(t)$ is well-defined on any finite interval $(0,T]$ and belongs to the space $C((0,T] : E),$ as well as that
$u(0)=u_{0}$ and
${\mathbf D}_{t}^{\gamma}u(t)
-f(t) \in  {\mathcal A}u(t)$ for $t>0.$
\end{defn}

Set $R_{\gamma}(t):= t^{\gamma -1}P_{\gamma}(t),$ $t>0.$ Then
we need the following lemma (cf. also \cite[Lemma 2.7]{hjms}):

\begin{lem}\label{andrade-wer}
Let $f\in AAPS^{q}([0,\infty) : X)$ with some $q\in (1,\infty),$ let $1/q+1/q'=1,$ and let 
$
q'(\gamma \beta -1)>-1.
$
Define 
\begin{align*}
H(t):=\int^{t}_{0}R_{\gamma}(t-s)f(s)\, ds,\quad t\geq 0.
\end{align*}
Then $H\in AAP([0,\infty) : X).$
\end{lem}

\begin{proof}
Let the locally $p$-integrable functions $g: {\mathbb R} \rightarrow X,$ 
$q: [0,\infty)\rightarrow X$ satisfy the conditions from \cite[Lemma 1]{hernan1}, with the number $p$ replaced with $q$ therein, and let the function $G(\cdot)$ be given by \eqref{wer}, with the function $f(\cdot)$ replaced by $g(\cdot)$ therein. Since $
q'(\gamma \beta -1)>-1
$  and, due to \eqref{debil},
$\|R_{\gamma}(t)\| \leq M_{1}t^{\gamma \beta-1},$ $t\in (0,1],$ we can apply
Lemma \ref{ravi-and} (see also Remark \ref{conclusion}) in order to see that $G(\cdot)$ is almost periodic. On the other hand, it is clear that there exists a number $\eta \in (0,1)$ such that $(1-\eta)(1+\gamma)>1.$
Put
\begin{align*}
F(t):=\int^{t}_{0}R_{\gamma}(t-s)q(s)\, ds-\int^{\infty}_{t}R_{\gamma}(s)g(t-s)\, ds,\quad t\geq 0.
\end{align*}
Owing to H\"older inequality, we have that $H(\cdot)$ is well-defined. Since $H(t)=G(t)+F(t)$ for all $ t\geq 0,$ we need to prove that $F\in C_{0}([0, \infty) : X).$  Evidently, $\|R_{\gamma}(t)\| \leq M_{2}t^{-\gamma -1},$ $t\geq 1$ and
\begin{align*}
\Biggl\|\int^{\infty}_{t}R_{\gamma}(s)& g(t-s)\, ds\Biggr\|\leq \sum_{k=0}^{\infty}\|R_{\gamma}(\cdot)\|_{L^{q'}[t+k,t+k+1]}  \|g\|_{S^{q}}
\\ & \leq \sum_{k=0}^{\infty}\|R_{\gamma}(\cdot)\|_{L^{\infty}[t+k,t+k+1]}  \|g\|_{S^{q}}
\leq \sum_{k=0}^{\infty}\frac{ \|g\|_{S^{q}}}{(t+k)^{\gamma +1}}
\\ & \leq \mbox{Const. } \sum_{k=0}^{\infty}\frac{ \|g\|_{S^{q}}}{t^{\eta (1+\gamma)}k^{(1-\eta)(\gamma +1)}}
\leq \mbox{Const. }t^{-\eta (1+\gamma)}\|g\|_{S^{q}},\quad t> 1.
\end{align*}
On account of this, we have that $\lim_{t\rightarrow \infty}\int^{\infty}_{t}R_{\gamma}(s) g(t-s)\, ds=0.$ 
Arguing similarly, we obtain that:
\begin{align*}
\Biggl\| \int^{t/2}_{0}R_{\gamma}(t-s)& q(s)\, ds\Biggr\| \leq \|g\|_{S^{q}} \sum_{k=0}^{\lceil t/2\rceil}\|R_{\gamma}(t-\cdot)\|_{L^{q'}[k,k+1]}
\\ & \leq \|g\|_{S^{q}} \sum_{k=0}^{\lceil t/2\rceil}\|R_{\gamma}(t-\cdot)\|_{L^{\infty}[k,k+1]} 
\\ & \leq M_{2}(1+\lceil t/2\rceil)(t-\lceil t/2\rceil)^{-\gamma -1}\|g\|_{S^{q}},\quad t\geq 2;
\end{align*}
hence, $\lim_{t\rightarrow \infty}\int^{t/2}_{0}R_{\gamma}(t-s)q(s)\, ds=0.$
It remains to be proved that\\ $\lim_{t\rightarrow \infty}\int^{t}_{t/2}R_{\gamma}(t-s)q(s)\, ds=0$ (observe that the integral in this limit expression converges by H\"older inequality, the estimate $
q'(\gamma \beta -1)>-1
$ and the $S^{q}$-boundedness of function $q(\cdot)$). Let a number
$\epsilon>0$ be fixed. Then there exists $t_{0}>0$ such that $\int^{t+1}_{t}\|q(s)\|^{q}\, ds<\epsilon^{q},$ $t\geq t_{0}.$
Suppose that $t>2t_{0}+6.$ Then the H\"older inequality implies the existence of a finite constant $c>0$ such that:
\begin{align*}
\Biggl\| \int^{t}_{t/2}R_{\gamma}(t-s)& q(s)\, ds\Biggr\| 
\\ &  \leq c\sum_{k=0}^{\lfloor t/2\rfloor-2}\|R_{\gamma}(t-\cdot)\|_{L^{q'}[t/2+k,t/2+k+1]}\epsilon +\epsilon \bigl \|\cdot^{\beta-1}\bigr\|_{L^{q'}[0,2]}
\\ &  \leq c\sum_{k=0}^{\lfloor t/2\rfloor-2}\|R_{\gamma}(t-\cdot)\|_{L^{\infty}[t/2+k,t/2+k+1]}\epsilon +\epsilon \bigl \|\cdot^{\beta-1}\bigr\|_{L^{q'}[0,2]}
\\ &  \leq c\epsilon M \sum_{k=0}^{\lfloor t/2\rfloor-2}(t/2+k)^{-\gamma-1}+\epsilon \bigl \|\cdot^{\beta-1}\bigr\|_{L^{q'}[0,2]}
\\ &  \leq c\epsilon M (t/2)^{-\eta (1+\gamma)} \sum_{k=0}^{\infty}k^{(1-\eta)(1+\gamma)}
+\epsilon \bigl \|\cdot^{\beta-1}\bigr\|_{L^{q'}[0,2]}.
\end{align*}
This estimate simply completes the proof of lemma.
\end{proof}

\begin{rem}\label{andrade-wer-remark}
Let $f\in AAPS^{q}([0,\infty) : X)$ with some $q\in (1,\infty),$ let $1/q+1/q'=1,$ and let 
$
q'(\gamma \beta -1)>-1.
$
Let $({\mathbf R}_{\gamma}(t))_{t>0}\subseteq L(X,Y)$ be a strongly continuous operator family, and let $H :[0,\infty) \rightarrow Y$ be defined as above.
Then $H\in AAP([0,\infty) : Y),$ provided that $({\mathbf R}_{\gamma}(t))_{t>0}$ has the same growth rate at zero and infinity as the operator family $(R_{\gamma}(t))_{t>0}$ considered above.
\end{rem}

Our first result reads as follows:

\begin{thm}\label{mambas-MLO-poy}
Suppose that $ 1\geq \theta>1-\beta$ and $ x_{0}\in D((-{\mathcal A})^{\theta}),$ resp. $ 1> \theta>1-\beta$
and $x_{0}\in X_{{\mathcal A}}^{\theta},$ as well as there exists a constant $\sigma >\gamma (1-\beta)$ such that, for every $T>0,$ there exists a finite constant $M_{T}>0$ such that $f : [0,\infty) \rightarrow X$ satisfies
\begin{align*}
\|f(t)-f(s)\| \leq M_{T}|t-s|^{\sigma},\quad 0\leq t,\ s\leq T.
\end{align*}
Let $ 1\geq \theta>1-\beta,$ resp. $ 1> \theta>1-\beta,$ and let
$$
f\in L^{\infty}_{loc}\Bigl((0,\infty) : \bigl[D\bigl(({-\mathcal A})^{\theta}\bigr)\bigr]\Bigr),\mbox{ resp. }f\in L^{\infty}_{loc}\Bigl((0,\infty) : X_{{\mathcal A}}^{\theta}\Bigr).
$$
Then there exists a unique classical solution $u(\cdot)$ of problem $\emph{(DFP)}_{f,\gamma}.$ If, additionally, $f\in AAPS^{q}([0,\infty) : X)$ with some $q\in (1,\infty),$ $1/q+1/q'=1$ and $q'(\gamma \beta -1)>-1,$
then  $u\in AAP([0,\infty) : X).$
\end{thm}

\begin{proof}
The first part of theorem is a simple consequence of \cite[Theorem 3.5.3]{FKP}, which also shows that the classical solution of (DFP)$_{f,\gamma}$
is given by the formula
$$
u(t)=S_{\gamma}(t)x_{0}+\int^{t}_{0}\bigl(t-s\bigr)^{\gamma-1}P_{\gamma}(t-s)f(s)\, ds,\quad t\geq 0.
$$
The second part of theorem immediately follows from this representation, Lemma \ref{nmb} and Lemma \ref{andrade-wer}.
\end{proof}

Suppose now that $ 1\geq \theta>1-\beta$ and $ x_{0}\in D((-{\mathcal A})^{\theta}),$ resp. $ 1> \theta>1-\beta$
and $x_{0}\in X_{{\mathcal A}}^{\theta}.$ In \cite[Section 3.5]{FKP}, we have proved that there exists a finite constant
$M_{1,\theta}>0$ such that
\begin{align}\label{debil-teta-p}
\bigl\| S_{\gamma}(t) \bigr\|_{L(X, [D((-{\mathcal A})^{\theta})])}+\bigl\| P_{\gamma}(t) \bigr\|_{L(X, [D((-{\mathcal A})^{\theta})])} \leq M_{1}t^{\gamma (\beta-\theta-1)},\quad t>0,
\end{align}
resp.
\begin{align}\label{debil-teta-cure-p}
\bigl\| S_{\gamma}(t) \bigr\|_{L(X,X_{\mathcal A}^{\theta})}+\bigl\| P_{\gamma}(t) \bigr\|_{L(X,X_{\mathcal A}^{\theta})}\leq M_{1}t^{\gamma (\beta-\theta-1)},\quad t>0.
\end{align}

We need to improve the estimates \eqref{debil-teta-p}-\eqref{debil-teta-cure-p} for long time behaviour:

\begin{lem}\label{nmb-frix}
There exists a finite constant $M_{2,\theta}>0$ such that
\begin{align*}
\bigl\| S_{\gamma}(t) \bigr\|_{L(X, [D((-{\mathcal A})^{\theta})])}+\bigl\| S_{\gamma}(t) \bigr\|_{L(X,X_{\mathcal A}^{\theta})}
\leq M_{2,\theta}t^{-\gamma (\theta+1)},\quad t>0,
\end{align*}
resp.
\begin{align*}
\bigl\| P_{\gamma}(t) \bigr\|_{L(X, [D((-{\mathcal A})^{\theta})])}+\bigl\| P_{\gamma}(t) \bigr\|_{L(X,X_{\mathcal A}^{\theta})}\leq M_{2,\theta}t^{-\gamma (\theta+1)-1},\quad t>0.
\end{align*}
\end{lem}

\begin{proof} 
Let $t>2/|c|.$ By the proof of \cite[Proposition 3.2]{faviniyagi} and the well known integral computation similar to that appearing in the proof of \cite[Lemma 7.1]{tsk}, we have that
$$
\Biggl(T(t)x,\frac{1}{2\pi i}\int_{\Gamma}(-\lambda)^{\theta}e^{\lambda t}\bigl(  \lambda -{\mathcal A} \bigr)^{-1}x\, d\lambda \Biggr) \in \bigl(-{\mathcal A}\bigr)^{\theta},\quad x\in X
$$
as well as
$$
\xi^{\theta}\Bigl[ \xi (\xi-{\mathcal A})^{-1}-I\Bigr]T(t)x =\frac{1}{2\pi i}\int_{\Gamma}\frac{\xi^{\theta} \lambda}{\xi-\lambda}e^{\lambda t}\bigl(  \lambda -{\mathcal A} \bigr)^{-1}x\, d\lambda,\quad \xi>0, \ x\in X.
$$
Let $a(t)>0$ satisfy $c^{2}(a(t)^{2}+1)+a(t)^{2}=t^{-2}.$
Using Cauchy theorem, we can deform the path of integration $\Gamma$ to the upwards oriented curve $\Gamma',$ obtained by replacing the union of segments $[c, c(a(t)+1)+ia(t)] \cup [c(-a(t)+1)-ia(t),c]$ of the curve $\Gamma$ with the part of circle with radius $1/t$ and center at point $c.$
Applying the computation contained in the proof of \cite[Theorem 2.6.1]{a43}, we get that
\begin{align*}
\bigl\| T(t) \bigr\|_{L(X, [D((-{\mathcal A})^{\theta})])}+\bigl\| T(t) \bigr\|_{L(X, [D((-{\mathcal A})^{\theta})])} \leq M_{2,\theta}'e^{ct}t^{\gamma (\beta-\theta-1)},
\end{align*}
resp.
\begin{align*}
\bigl\| T(t) \bigr\|_{L(X,X_{\mathcal A}^{\theta})}+\bigl\| T(t) \bigr\|_{L(X,X_{\mathcal A}^{\theta})}\leq M_{2,\theta}'e^{ct}t^{\gamma (\beta-\theta-1)}.
\end{align*}
Now the final conclusions follow similarly as in the proof of Lemma \ref{nmb}.
\end{proof}

Using Lemma \ref{nmb-frix} and Remark \ref{andrade-wer-remark}, we can simply prove that the following holds true:

\begin{thm}\label{mambas-MLO-poy-frac}
Let the requirements of the first part of \emph{Theorem \ref{mambas-MLO-poy}} hold. Then there exists a unique classical solution $u(\cdot)$ of problem $\emph{(DFP)}_{f,\gamma}.$
Assume, further, that the initial value $x_{0}$ satisfies $\lim_{t\rightarrow 0+}S_{\gamma}(t)x_{0}=x_{0}$ in $ [D((-{\mathcal A})^{\theta})],$ resp. $X_{{\mathcal A}}^{\theta},$
as well as $f\in AAPS^{q}([0,\infty) : X)$ and $1>\theta>1-\beta,$ resp. $f\in AAPS^{q}([0,\infty) : X),$
with some $q\in (1,\infty).$ Let $1/q+1/q'=1$ and $q'(\gamma (\beta-\theta) -1)>-1.$
Then  $u\in AAP([0,\infty) : [D((-{\mathcal A}^{\theta}))]),$ resp. $u\in AAP([0,\infty) :  X_{{\mathcal A}}^{\theta}).$
\end{thm}

\begin{rem}\label{eksponent-beta}
\begin{itemize}
\item[(i)]
The condition $q'(\gamma (\beta-\theta) -1)>-1$ immediately forces that $\beta -\gamma >0$ and therefore $\beta >1/2,$ which can be slightly restrictive in applications (see \cite[Section 3.7]{faviniyagi} for more details). 
\item[(ii)] It is very simple to see that $\lim_{t\rightarrow 0+}S_{\gamma}(t)x_{0}=x_{0}$ in $ [D((-{\mathcal A})^{\theta})],$ resp. $X_{{\mathcal A}}^{\theta},$ holds provided that 
$\lim_{t\rightarrow 0+}T(t)x_{0}=x_{0}$ in $ [D((-{\mathcal A})^{\theta})],$ resp. $X_{{\mathcal A}}^{\theta}.$ Concerning the space $ [D((-{\mathcal A})^{\theta})],$ we have that the last equality holds for any $x_{0}\in (-{\mathcal A})^{-\theta}(\Omega_{c}),$
where $\Omega_{c}$ denotes the domain of continuity of semigroup $(T(t))_{t>0};$ speaking-matter-of-factly, if $y_{0}\in \Omega_{c}$ and $x_{0}=(-{\mathcal A})^{-\theta}y_{0},$ then
$T(t)y_{0}-y_{0}\in (-{\mathcal A})^{\theta}(T(t)x_{0}-x_{0}),$ $t>0$ and therefore $\| T(t)x_{0}-x_{0}\|_{[D((-{\mathcal A})^{\theta})]}\leq \| T(t)y_{0}-y_{0} \|,$ $t>0,$ which implies the claimed assertion. 
 Concerning the space $X_{{\mathcal A}}^{\theta},$ the situation is much more complicated and, without any doubt, not well explored in the existing literature.
\end{itemize}
\end{rem}

For semilinear problems, we will use the following notion.

\begin{defn}\label{defin-ska-MLOs}
By a mild solution of (DFP)$_{f,\gamma,s},$ we mean any function $u\in C([0,\infty) : X)$ satisfying that 
$$
u(t)=S_{\gamma}(t)x_{0}+\int^{t}_{0}\bigl(t-s\bigr)^{\gamma-1}P_{\gamma}(t-s)f(s,u(s))\, ds,\quad t\geq 0.
$$
\end{defn}

Set, for every $x\in C_{b}([0,\infty) :X),$
$$
(\Upsilon x)(t):= S_{\gamma}(t)x_{0}+\int^{t}_{0}\bigl(t-s\bigr)^{\gamma-1}P_{\gamma}(t-s)f(s,x(s))\, ds,\quad t\geq 0.
$$

Suppose that \eqref{vbnmp} holds for a.e. $t>0$ ($I=[0,\infty)$), with locally integrable positive function $L_{f}(\cdot).$ Set, for every $n\in {\mathbb N},$
\begin{align*}
A_{n}:=&\sup_{t\geq 0}\int^{t}_{0}\int^{x_{n}}_{0}\cdot \cdot \cdot \int^{x_{2}}_{0}
\bigl\|R_{\gamma}(t-x_{n})\bigr\|
\\& \times \prod^{n}_{i=2}\bigl\|R_{\gamma}(x_{i}-x_{i-1})\bigr\| \prod^{n}_{i=1}L_{f}(x_{i})\, dx_{1}\, dx_{2}\cdot \cdot \cdot \, dx_{n}.
\end{align*}
Then a simple calculation shows that
\begin{align}\label{storn-vuk}
\Bigl \| \bigl(\Upsilon^{n} u\bigr)-\bigl(\Upsilon^{n} v\bigr)\Bigr\|_{\infty}\leq A_{n}\bigl\| u- v\bigr\|_{\infty},\quad u,\ v\in BUC([0,\infty) : X),\ n\in {\mathbb N}.
\end{align}

The following result is in a close relationship with \cite[Theorem 2.10]{hjms}:

\begin{thm}\label{stepa-vuk}
Suppose that $I=[0,\infty)$ and the following conditions hold:
\begin{itemize}
\item[(i)] $g \in APS^{p}(I \times X : X)  $ with  $p > 1, $ and there exist a number $ r\geq \max (p, p/p -1)$ and a function $ L_{g}\in L_{S}^{r}(I:X) $ such that \eqref{vbnmp} holds.
\item[(ii)] $f(t,x)=g(t,x)+q(t,x)$ for all $t\geq 0$ and $x\in X,$ where $\hat{q}\in C_{0}(I \times X : L^{q}([0,1]:X))$
and $q=pr/p+r.$

Set 
\begin{align*}
q':=\infty,\mbox{ provided }r=p/p-1 \mbox{ and }q':=\frac{pr}{pr-p-r},\mbox{ provided }r>p/p-1.
\end{align*}

Assume also that:
\item[(iii)] 
$
q'(\gamma \beta -1)>-1, 
$
\item[(iv)] 
\emph{\eqref{vbnmp}} holds for a.e. $t>0$, with  locally bounded positive function $L_{f}(\cdot)$ satisfying
$A_{n}<1$ for some $n\in {\mathbb N}.$
\end{itemize}
Then there exists a unique asymptotically almost periodic solution of inclusion $\emph{(DFP)}_{f,\gamma,s}.$
\end{thm}

\begin{proof}
Due to (i)-(ii) and Lemma \ref{bibl}, we get that for each $x\in AAP(I:X)$ one has $f(\cdot, x(\cdot)) \in AAPS^{q}(I : X),$ where $q=pr/p+r;$ here we would like to recall only that the range of an $X$-valued asymptotically almost periodic function is relatively compact in $X$ by \cite[Theorem 2.4]{zhang-c-prim}.
Owing to the assumption (iii), Lemma \ref{andrade-wer} and the obvious equality $\lim_{t\rightarrow +\infty}S_{\gamma}(t)x_{0}=0$, we get that the mapping $
\Upsilon : AAP(X) \rightarrow AAP(X)$ is well-defined.
Using now \eqref{storn-vuk}, (iv) and a well-known extension of the Banach contraction principle, we obtain the existence of an asymptotically almost periodic solution of inclusion (DFP)$_{f,\gamma,s}.$ To prove the uniqueness of mild solutions, assume that
$u(\cdot)
$ and $
v(\cdot)
$ are two mild solutions of  inclusion (DFP)$_{f,\gamma,s}.$ Then \eqref{debil} yields that
\begin{align*}
\| u(t)-v(t) \| & \leq M_{1}\int^{t}_{0}\bigl(t-s\bigr)^{\gamma \beta -1}L_{f}(s)\| u(s)-v(s) \|\, ds,
\quad t\geq 0.
\end{align*}
By \cite[Lemma 6.19, p. 111]{Diet}, we get that $ u(s)=v(s)$  for all $s\in [ 0,t]$  ($t>0$ fixed). This completes the proof of theorem.
\end{proof}

If we employ Lemma \ref{bibl-frimaonji} in place of Lemma \ref{bibl}, then we are in a position to formulate and prove the following analogue of Theorem \ref{stepa-vuk} in the case of consideration of classical Lipschitz condition \eqref{vbnmp-frim}:

\begin{thm}\label{stepa-vuk-frimex}
Let $I =[0,\infty),$ and let $p>1.$
Suppose that the following conditions hold:
\begin{itemize}
\item[(i)] $g \in APS^{p}(I \times X : X)  $ with  $p \geq 1, $ and there exists a constant $L>0$ such that \eqref{vbnmp-frim} holds with the function $f(\cdot, \cdot)  $ replaced by the function $g(\cdot, \cdot)  $ therein.
\item[(ii)] $f(t,x)=g(t,x)+q(t,x)$ for all $t\geq 0$ and $x\in X,$ where $\hat{q}\in C_{0}(I \times X : L^{p}([0,1]:X)).$
\item[(iii)] $
\frac{p}{p-1}(\gamma \beta -1)>-1.$
\item[(iv)] 
\emph{\eqref{vbnmp}} holds for a.e. $t>0$, with  locally bounded positive function $L_{f}(\cdot)$ satisfying
$A_{n}<1$ for some $n\in {\mathbb N}.$
\end{itemize}
Then there exists a unique asymptotically almost periodic solution of inclusion $\emph{(DFP)}_{f,\gamma,s}.$
\end{thm}

Since any function $g \in AAP(I \times X : X)  $ 
belongs to the class
 $AAPS^{p}(I \times X : X)  $ for all $p>1$, and 
$$
\lim_{p\rightarrow +\infty} \frac{p}{p-1}(\gamma \beta -1)=\gamma \beta -1>-1,
$$
Theorem \ref{stepa-vuk-frimex} immediately implies
the following important corollary:

\begin{cor}\label{kretinjo-vuk}
Suppose that $I=[0,\infty),$ the function $f(\cdot,\cdot)$ is asymptotically almost periodic and \emph{\eqref{vbnmp}} holds for a.e. $t>0$, with  locally bounded positive function $L_{f}(\cdot)$ satisfying
$A_{n}<1$ for some $n\in {\mathbb N}.$ Then there exists a unique asymptotically almost periodic solution of inclusion $\emph{(DFP)}_{f,\gamma,s}.$
\end{cor}

It is not trivial to state a satisfactory criterion which would enable one to see that there exists an integer $n\in {\mathbb N}$ such that
$A_{n}<1.$ On the other hand, a very simple calculation involving the estimates \eqref{debil} and \eqref{debil-prim} shows that 
$$
A_{1}\leq L\Biggl[\frac{M_{1}}{\gamma \beta}+\frac{M_{2}}{\gamma}\Biggr],
$$
provided the requirements of Corollary \ref{kretinjo-vuk}.
Hence, we have the following:

\begin{cor}\label{kret-vuk}
Suppose that $I=[0,\infty),$ the function $f(\cdot,\cdot)$ is asymptotically almost periodic and \emph{\eqref{vbnmp-frim}} holds for some $L\in [ 0, (\frac{M_{1}}{\gamma \beta}+\frac{M_{2}}{\gamma})^{-1} ).$ Then there exists a unique asymptotically almost periodic solution of inclusion $\emph{(DFP)}_{f,\gamma,s}.$
\end{cor}

Now we would like to present the following illustrative example (see \cite{hjms} for more details).

\begin{example}\label{merkle-nedegenerisane-vuk} 
\begin{itemize} 
\item[(i)] (von Wahl \cite{wolf}) We can simply incorporate Theorem \ref{stepa-vuk}-Theorem \ref{stepa-vuk-frimex} in the analysis of existence and uniqueness of asymptotically almost periodic solutions of the following 
fractional semilinear equation with higher order differential operators in the H\"older space $X=C^{\alpha}(\overline{\Omega}):$
\[\left\{
\begin{array}{l}
{\mathbf D}_{t}^{\gamma}u(t,x)=-\sum \limits_{|\beta|\leq 2m}a_{\beta}(t,x)D^{\beta}u(t,x)-\sigma u(t,x) +f(t,u(t,x)),\ t\geq 0,\ x\in {\Omega};\\
u(0,x)=u_{0}(x),\quad x\in {\Omega},
\end{array}
\right.
\]
where $\alpha\in (0,1),$ $m\in {\mathbb N},$
$\Omega$ is a bounded domain in ${{\mathbb R}^{n}}$ with boundary of
class $C^{4m}$, 
$D^{\beta}=\prod_{i=1}^{n}(\frac{1}{i}\frac{\partial}{\partial
x_{i}})^{\beta_{i}},$ the functions $a_{\beta} : \overline{\Omega} \rightarrow
{\mathbb C}$ satisfy certain conditions and $\sigma>0$ is sufficiently large. 
Then the condition (P) holds with the exponent $\beta=1-\frac{\alpha}{2m}$ and this value is known to be sharp.
\item[(ii)] (Favini, Yagi \cite[Example 3.6]{faviniyagi}) 
Concerning degenerate semilinear fractional differential equations, it is clear that the possible applications of Theorem \ref{stepa-vuk}-Theorem \ref{stepa-vuk-frimex} can be given in the analysis of existence and uniqueness of asymptotically almost periodic solutions of the following 
fractional Poisson semilinear heat equation in the Lebesgue space $X=L^{p}(\Omega):$
\[\left\{
\begin{array}{l}
{\mathbf D}_{t}^{\gamma}[m(x)v(t,x)]=(\Delta -b )v(t,x) +f(t,m(x)v(t,x)),\ t\geq 0,\ x\in {\Omega};\\
v(t,x)=0,\quad (t,x)\in [0,\infty) \times \partial \Omega ,\\
 m(x)v(0,x)=u_{0}(x),\quad x\in {\Omega},
\end{array}
\right.
\]
where $\Omega$ is a bounded domain in ${\mathbb R}^{n},$ $b>0,$ $m(x)\geq 0$ a.e. $x\in \Omega$, $m\in L^{\infty}(\Omega),$ $1<p<\infty$ and the operator $\Delta -b $ acts on $X$ with the Dirichlet boundary conditions.
\end{itemize}
\end{example}

\subsection{The non-analyticity of semigroup $(T(t))_{t>0}$}\label{urban}

It is not difficult to observe that the analyticity of degenerate semigroup $(T(t))_{t>0}$ examined in the previous subsection is a slightly redundant assumption.
In this subsection, we investigate the case in which the operator $C\in L(X)$ is injective and $(T(t))_{t\geq 0}\subseteq L(X)$ is a $C$-regularized semigroup, i.e, the mapping $t\mapsto T(t)x,$ $t\geq 0$ is continuous for every fixed element $x\in X,$ $T(0)=C$ and $T(t+s)C=T(t)T(s)$ for all $t,\ s\geq 0.$ The multivalued linear operator 
$$
{\mathcal A}:=\Biggl\{ (x,y)\in X\times X : T(t)x-Cx=\int^{t}_{0}T(s)y\, ds \mbox{ for all }
t\geq 0 \Biggr\}
$$
is said to be the integral generator of $(T(t))_{t\geq 0};$ see \cite{knjigah}-\cite{FKP} for more details on the subject. Let $\|T(t)\|=O(e^{ct}),$ $t\geq 0$ for some negative constant $c<0.$ Then we know that $\{\lambda \in {\mathbb C} : \Re \lambda>c\}\subseteq \rho_{C}({\mathcal A})$ as well as $\int^{\infty}_{0}e^{-\lambda t}T(t)x\, dt = (\lambda-{\mathcal A})^{-1}C,$ $x\in X,$ $\Re \lambda>c$ (\cite{FKP}).

Define the operator families $S_{\gamma}(\cdot),$ $P_{\gamma}(\cdot)$ and $R_{\gamma}(\cdot)$ as well as the sequence $(A_{n})_{n\in {\mathbb N}}$ and the operator $\Upsilon (\cdot)$ as in the previous subsection. 
Then the operator ${\mathcal A}=C^{-1}{\mathcal A}C$ is the integral generator of an exponentially bounded $(g_{\gamma},C)$-regularized resolvent family $(S_{\gamma}(t))_{t\geq 0}$ (cf. \cite[Definition 3.2.2]{FKP} for the notion), so that $
{\mathcal A}$ is closed,
$(S_{\gamma}(t))_{t\geq 0}$ is strongly continuous for $t\geq 0,$
$S_{\gamma}(t)x-Cx\in {\mathcal A}\int^{t}_{0}g_{\gamma}(t-s)S_{\gamma}(s)x\, ds,$ $t\geq 0,$ $x\in X;$ $S_{\gamma}(t)x-Cx=\int^{t}_{0}g_{\gamma}(t-s)S_{\gamma}(s)y\, ds,$ $t\geq 0,$ whenever $y\in {\mathcal A}x;$
$S_{\gamma}(t){\mathcal A} \subseteq {\mathcal A}S_{\gamma}(t),$ $t\geq 0$ and $S_{\gamma}(t)C=CS_{\gamma}(t),$ $t\geq 0.$

In this subsection, we will employ the classical definition of Caputo fractional derivatives of order $\gamma \in (0,1)$. 
The Caputo fractional derivative\index{fractional derivatives!Caputo}
${\mathbf D}_{t}^{\gamma}u(t)$ is defined for those continuous functions
$u :[0,\infty) \rightarrow X$ for which 
$g_{1-\gamma}\ast (u(\cdot)-u(0)) \in C^{1}([0,\infty) : X),$
by
$$
{\mathbf
D}_{t}^{\gamma}u(t)=\frac{d}{dt}\Biggl[g_{1-\gamma}
\ast \Bigl(u(\cdot)-u(0)\Bigr)\Biggr](t),\quad t\geq 0.
$$

We continue by observing that the arguments contained in the proof of Lemma \ref{nmb} show that there exist two finite constants $M_{3}>0$ and $M_{4}>0$ such that
\begin{align}\label{debil-vuk}
\bigl\| S_{\gamma}(t) \bigr\|+\bigl\| P_{\gamma}(t) \bigr\|\leq M_{3},\quad t>0
\end{align}
as well as
\begin{align}\label{debil-prim-vuk}
\bigl\| S_{\gamma}(t) \bigr\|\leq M_{4}t^{-\gamma},\ t\geq 1 \ \ \mbox{  and  }\ \ \bigl\| P_{\gamma}(t) \bigr\|\leq M_{4}t^{-2\gamma },\ t\geq 1.
\end{align}

\begin{defn}\label{grouna}
Let $x_{0}\in R(C),$ and let $f: [0,\infty) \rightarrow X$ be continuous. Then we say that a continuous function $u: [0,\infty) \rightarrow X$ is a classical solution of (DFP)$_{f,\gamma}$
iff $
{\mathbf
D}_{t}^{\gamma}u(t)$ is well-defined and continuous for $t\geq 0$ as well as (DFP)$_{f,\gamma}$ holds identically for $t\geq 0.$
\end{defn}

Now we state the following result:

\begin{thm}\label{mambas-MLO-poyzd}
Suppose that $C^{-1}x_{0}\in D({\mathcal A})$ and $C^{-1}f \in W^{1,1}_{loc}((0,\infty) : X).$ 
Then there exists a unique classical solution $u(\cdot)$ of problem $\emph{(DFP)}_{f,\gamma}.$ If, additionally, $f\in AAPS^{q}([0,\infty) : X)$ with some $q\in (1,\infty),$ $1/q+1/q'=1$ and $q'(\gamma -1)>-1,$
then  $u\in AAP([0,\infty) : X).$
\end{thm}

\begin{proof} 
The uniqueness of classical solutions is a simple consequence of the fact that 
${\mathcal A}$ generates a global $(g_{\gamma},C)$-regularized resolvent family (\cite{FKP}). To prove the existence of classical solutions, we first observe that
the argumentation contained in the proof of \cite[Theorem 5]{warma} (see also \cite{iasi}) shows that $\int^{\infty}_{0}e^{-\lambda t}R_{\gamma}(t)x\, dt = (\lambda^{\gamma}-{\mathcal A})^{-1}C,$ $x\in X,$ $\lambda>0.$ By the uniqueness theorem for Laplace transform, we get that
\begin{align}\label{milanche}
\bigl(g_{1-\gamma} \ast R_{\gamma}(\cdot)x\bigr)(t)=S_{\gamma}(t)x,\quad t\geq 0,\ x\in X.
\end{align}  
Let $C^{-1}x_{0}=y_{0}$ and $z_{0}\in {\mathcal A}y_{0}.$ Then $S_{\gamma}(t)y_{0}-Cy_{0}=\int^{t}_{0}g_{\gamma}(t-s)S_{\gamma}(s)z_{0}\, ds,$ $t\geq 0$ and, on account of this, it is almost trivial to verify that $
{\mathbf
D}_{t}^{\gamma}S_{\gamma}(t)C^{-1}x_{0}=S_{\gamma}(t)z_{0},$ $t\geq 0.$ Set $F_{\gamma}(t):=\int^{t}_{0}R_{\gamma}(t-s)C^{-1}f(s)\, ds,$ $t\geq 0$ and $u(t):=S_{\gamma}(t)C^{-1}x_{0}+F_{\gamma}(t),$ $t\geq 0.$
Using the proof of \cite[Proposition 1.3.6]{a43}, we get that
\begin{align}\label{milanche-krki}
F_{\gamma}^{\prime}(t)=\bigl( R_{\gamma} \ast (C^{-1}f)^{\prime}\bigr)(t)+R_{\gamma}(t)(C^{-1}f)(0),\quad t> 0,\ x\in X,
\end{align}  
which simply implies that $
{\mathbf
D}_{t}^{\gamma}F_{\gamma}(t)=(g_{1-\gamma} \ast F_{\gamma}^{\prime})(t),$ $t\geq 0.$ By the foregoing, it suffices to show that
\begin{align*}
\bigl(g_{1-\gamma} \ast F_{\gamma}^{\prime}\bigr)(t) -f(t)\in {\mathcal A} \bigl[ R_{\gamma} \ast C^{-1}f\bigr](t),\quad t\geq 0,
\end{align*}
i.e., by \eqref{milanche}-\eqref{milanche-krki},
\begin{align*}
\bigl(S_{\gamma} & \ast (C^{-1}f)^{\prime}\bigr)(t)+S_{\gamma}(t)C^{-1}f(0)-f(t)
\\ &=\bigl(S_{\gamma} \ast C^{-1}f\bigr)^{\prime}(t)-f(t)\in {\mathcal A} \bigl[ R_{\gamma} \ast C^{-1}f\bigr](t),\quad t\geq 0.
\end{align*}
Due to the closedness of ${\mathcal A}$, the only thing that remained to be proved is 
\begin{align*}
\Bigl(g_{1}\ast\bigl(S_{\gamma} \ast C^{-1}f\bigr)^{\prime}\Bigr)(t)-\bigl(g_{1}\ast f\bigr)(t)\in {\mathcal A} \bigl[ g_{1}\ast R_{\gamma} \ast C^{-1}f\bigr](t),\quad t\geq 0,
\end{align*}
i.e., due to \eqref{milanche},
\begin{align*}
\bigl(S_{\gamma} & \ast C^{-1}f\bigr)(t)-\bigl(g_{1}\ast f\bigr)(t)\in {\mathcal A} \bigl[ g_{\gamma}\ast S_{\gamma} \ast C^{-1}f\bigr](t),\quad t\geq 0.
\end{align*}
This follows from Lemma \ref{integracija-tricky} and the inclusion $S_{\gamma}(t)x-Cx\in {\mathcal A}\bigl( g_{\gamma}\ast S_{\gamma}(\cdot)x \bigr)(t),$ $t\geq 0,$ $x\in X.$
The remaining part of theorem is a consequence of Lemma \ref{andrade-wer}.
\end{proof}

 Following our analyses from \cite[Subsection 2.2.5]{FKP}, it will be convenient to introduce the following definition:

\begin{defn}\label{sgamma}
We say that a continuous function $t\mapsto u(t),$ $t\geq 0$ is a 
mild solution of the semilinear fractional Cauchy inclusion (DFP)$_{f,\gamma,s}$ iff 
$$
u(t)=S_{\gamma}(t)C^{-1}x_{0}+\int^{t}_{0}(t-s)^{\gamma-1}P_{\gamma}(t-s)C^{-1}f(s,u(s))\, ds,\quad t\geq 0.
$$
\end{defn}

Keeping in mind the estimates \eqref{debil-vuk}-\eqref{debil-prim-vuk} and the foregoing arguments, it is almost straightforward to formulate and prove the following analogues of 
Theorem \ref{stepa-vuk}-Theorem \ref{stepa-vuk-frimex} and Corollary \ref{kretinjo-vuk}-Corollary \ref{kret-vuk}:

\begin{thm}\label{stepa-vukk}
Suppose that $I=[0,\infty)$ and the following conditions hold:
\begin{itemize}
\item[(i)] $g \in APS^{p}(I \times X : X)  $ with  $p > 1, $ and there exist a number $ r\geq \max (p, p/p -1)$ and a function $ L_{g}\in L_{S}^{r}(I:X) $ such that \eqref{vbnmp} holds.
\item[(ii)] $C^{-1}f(t,x)=g(t,x)+q(t,x)$ for all $t\geq 0$ and $x\in X,$ where $\hat{q}\in C_{0}(I \times X : L^{q}([0,1]:X))$
and $q=pr/p+r.$

Set 
\begin{align*}
q':=\infty,\mbox{ provided }r=p/p-1 \mbox{ and }q':=\frac{pr}{pr-p-r},\mbox{ provided }r>p/p-1.
\end{align*}

Assume also that:
\item[(iii)] 
$
q'(\gamma -1)>-1, 
$
\item[(iv)] 
\emph{\eqref{vbnmp}} holds with functions $f(\cdot,\cdot)$ and $L_{f}(\cdot)$ replaced therein by the functions $C^{-1}f(\cdot,\cdot)$ and $L_{C^{-1}f}(\cdot)$, respectively, for a.e. $t>0$, where $L_{C^{-1}f}(\cdot)$ is a locally bounded positive function satisfying
$A_{n}<1$ for some $n\in {\mathbb N}.$
\end{itemize}
Then there exists a unique asymptotically almost periodic solution of inclusion $\emph{(DFP)}_{f,\gamma,s}.$
\end{thm}

\begin{thm}\label{stepa-vuk-frimexx}
Let $I =[0,\infty),$ and let $p>1.$
Suppose that the following conditions hold:
\begin{itemize}
\item[(i)] $g \in APS^{p}(I \times X : X)  $ with  $p \geq 1, $ and there exists a constant $L>0$ such that \eqref{vbnmp-frim} holds with the function $f(\cdot, \cdot)  $ replaced by the function $g(\cdot, \cdot)  $ therein.
\item[(ii)] $C^{-1}f(t,x)=g(t,x)+q(t,x)$ for all $t\geq 0$ and $x\in X,$ where $\hat{q}\in C_{0}(I \times X : L^{p}([0,1]:X)).$
\item[(iii)] $
\frac{p}{p-1}(\gamma -1)>-1.$
\item[(iv)] 
\emph{\eqref{vbnmp}} holds with functions $f(\cdot,\cdot)$ and $L_{f}(\cdot)$ replaced therein by the functions $C^{-1}f(\cdot,\cdot)$ and $L_{C^{-1}f}(\cdot)$, respectively, for a.e. $t>0$, where $L_{C^{-1}f}(\cdot)$ is a locally bounded positive function satisfying
$A_{n}<1$ for some $n\in {\mathbb N}.$
\end{itemize}
Then there exists a unique asymptotically almost periodic solution of inclusion $\emph{(DFP)}_{f,\gamma,s}.$
\end{thm}

\begin{cor}\label{kretinjo-vukk}
Suppose that $I=[0,\infty),$ the function $C^{-1}f(\cdot,\cdot)$ is asymptotically almost periodic and \emph{\eqref{vbnmp}} holds with functions $f(\cdot,\cdot)$ and $L_{f}(\cdot)$ replaced therein by the functions $C^{-1}f(\cdot,\cdot)$ and $L_{C^{-1}f}(\cdot)$, respectively, for a.e. $t>0$, where $L_{C^{-1}f}(\cdot)$ is a locally bounded positive function satisfying
$A_{n}<1$ for some $n\in {\mathbb N}.$ Then there exists a unique asymptotically almost periodic solution of inclusion $\emph{(DFP)}_{f,\gamma,s}.$
\end{cor}

A very simple calculation involving the estimates \eqref{debil}-\eqref{debil-prim} shows that 
$$
A_{1}\leq L\Biggl[\frac{M_{3}}{\gamma }+\frac{M_{4}}{\gamma}\Biggr],
$$
provided the requirements of Corollary \ref{kretinjo-vuk}.
Hence, we have the following:

\begin{cor}\label{kret-vukk}
Suppose that $I=[0,\infty),$ the function $C^{-1}f(\cdot,\cdot)$ is asymptotically almost periodic and \emph{\eqref{vbnmp-frim}} holds for some $L\in [ 0, (\frac{M_{3}}{\gamma }+\frac{M_{4}}{\gamma})^{-1} ).$ Then there exists a unique asymptotically almost periodic solution of inclusion $\emph{(DFP)}_{f,\gamma,s}.$
\end{cor}

Theoretical results established here can be applied to a class of abstract degenerate fractional differential equations, see e.g. \cite[Example 2.1]{faviniyagi} and \cite[Example 3.2.14]{FKP} ($C=I$). 
The most important applications with $C\neq I$ can be given to abstract non-degenerate fractional differential equations:

\begin{example}\label{pos-kos-ros}
Assume that $k\in\mathbb{N}$, $a_{\alpha}\in\mathbb{C}$, $0\!\leq\! |\alpha|\!\leq\! k$, $a_{\alpha}\neq 0$
for some $\alpha$ with $|\alpha|=k$, $P(x)=\sum_{|\alpha|\leq k}a_{\alpha}i^{|\alpha|}x^{\alpha}$,
$x\in\mathbb{R}^n$, $c':=\sup_{x\in\mathbb{R}^n}\Re(P(x))<0$,
$X$ is one of the spaces $L^p(\mathbb{R}^n)$ ($1\leq p\leq\infty$), $C_0(\mathbb{R}^n)$, $C_b(\mathbb{R}^n)$,
$BUC(\mathbb{R}^n)$,
$$
P(D):=\sum_{|\alpha|\leq k}a_{\alpha}f^{(\alpha)}
\text{ and } D(P(D)):=\bigl\{f\in E:P(D)f\in E\text{ distributionally}\bigr\}.
$$
Set $n_X:=n|\frac{1}{2}-\frac{1}{p}|$, if $X=L^p(\mathbb{R}^n)$
for some $p\in(1,\infty)$ and $n_X>\frac{n}{2}$, otherwise.
Then it is well known (see e.g. \cite[Example 2.8.6]{knjigah}) that the operator $P(D)$ generates a global $(1-\Delta)^{-n_Xk/2}$-regularized semigroup $(T(t))_{t\geq 0}$
satisfying the estimate $\|T(t)\|=O(e^{-ct}),$ $t\geq 0$ for all $c\in (c',0).$
This immediately implies that we can consider the existence and uniqueness of asymptotically almost periodic solutions of the following semilinear Cauchy problem:
\[\left\{
\begin{array}{l}
{\mathbf D}_{t}^{\gamma}u(t,x)=\sum_{|\alpha|\leq k}a_{\alpha}f^{(\alpha)}(t,x)+f(t,u(t,x)),\ t\geq 0,\ x\in \mathbb{R}^n,\\
 u(0,x)=u_{0}(x),\quad x\in \mathbb{R}^n.
\end{array}
\right.
\]
\end{example}

We close the paper by providing basic information on asymptotically almost periodic solutions of abstract fractional inclusion
\begin{align}\label{li-peng}
D^{\gamma}_{+}u(t)\in {\mathcal A}u(t)+f(t),\quad t\in {\mathbb R},
\end{align}
where $D^{\gamma}_{+}u(t)$ denotes
the Weyl-Liouville fractional derivative of order $\gamma$ and $f : {\mathbb R} \rightarrow X$ (\cite{relaxation-peng}). In this paper,
Mu, Zhoa and Peng have considered various types of (asymptotically) generalized almost periodic and generalized almost automorphic solutions of \eqref{li-peng} provided that the operator ${\mathcal A}$ is single-valued and generates an exponentially decaying $C_{0}$-semigroup.
By applying the Fourier transform in \cite[Lemma 6]{relaxation-peng}, the authors have proposed the following definition of mild solution of \eqref{li-peng}: A continuous function $u: {\mathbb R} \rightarrow X$ is said to be a mild solution of \eqref{li-peng} iff it has the following form
$$
u(t)=\int^{t}_{-\infty}(t-s)^{\gamma -1}P_{\gamma}(t-s)f(s)\, ds,\quad t\in {\mathbb R};
$$
a semilinear analogue is
$$
u(t)=\int^{t}_{-\infty}(t-s)^{\gamma -1}P_{\gamma}(t-s)f(s,u(s))\, ds,\quad t\in {\mathbb R}.
$$
We would like to observe that the method proposed in the proof of \cite[Lemma 6]{relaxation-peng} is completely meaningful in the case that  ${\mathcal A}$ is an MLO generating degenerate semigroup with removable singularity at zero or that
${\mathcal A}$ is an MLO generating an exponentially decaying $C$-regularized semigroup. Therefore, we are in a position to analyze the existence and uniqueness of asymptotically almost periodic solutions ${\mathbb R} \mapsto X$ of (semilinear) fractional Cauchy inclusion \eqref{li-peng} (cf. 
Zhang \cite{zhang-c-prim} for the notion) by using similar arguments to those employed in the proofs of \cite[Theorem 8, Theorem 9, Theorem 17, Theorem 19]{relaxation-peng}. Details can be left to the interested reader.

\end{document}